\def\seq#1_#2{\langle #1_#2:#2\in\omega\rangle}
\def\fc#1|#2{#1\uparrow#2}
\def\ain{\subseteq^*}
\def\cl#1{\overline{#1}}
\def\N{{\mathbb N}}
\def\Q{{\mathbb Q}}
\def\R{{\mathbb R}}

\def\C{{\mathfrak c}}
\def\Xo{{\bf 1}_X}
\def\set#1:#2.{{\{\,#1:#2\,\}}}
\def\Int{\mathop{\rm Int}}
\def\tN{{[\N]^{<\omega}}}
\def\tNN{{\tN\otimes\tN}}
\def\pQ{{pseudo-\hbox{$\Q$}}}
\def\so{\mathop{\rm so}}
\providecommand{\keywords}[1]{\textbf{\textit{Keywords:}} #1}
\documentclass{article}
\title{On sequential analytic groups}
\author{Alexander Y.~Shibakov\footnote{Tennessee Tech.\ Universiy,
email: {\tt ashibakov@tntech.edu}}}
\usepackage{amssymb}
\usepackage{amsthm}
\usepackage{amsmath}
\newtheorem{theorem}{Theorem}
\newtheorem{lemma}{Lemma}
\newtheorem{example}{Example}
\newtheorem{corollary}{Corollary}
\newtheorem{remark}{Remark}
\newtheorem{definition}{Definition}
\newtheorem{question}{Question}
\begin{document}
\maketitle
\begin{abstract}
\noindent We answer a question of S.~Todor\v cevi\'c and
C.~Uzc\'ategui from \cite{TU1} by showing that the only possible
sequential orders of sequential analytic groups are $1$ and
$\omega_1$. Other results on the structure of sequential analytic spaces
and their relation to other classes of spaces are given as well. In
particular, we provide a full topological classification of sequential
analytic groups by showing that all such groups are either metrizable
or $k_\omega$-spaces, which, together with a result by Zelenyuk,
implies that there are exactly $\omega_1$ non homeomorphic analytic
sequential group topologies.
\end{abstract}
\keywords{analytic space, topological group, sequential space}
\section{Introduction}
Spaces with `definable' topologies are ubiquitous in mathematics. They
often appear as examples when the topology construction does not use
the axiom of choice and frequently show up inside function
spaces (see \cite{TU1} for references and further motivation). To make
the notion of `definable' more precise recall that a family of subsets
of some countable set $X$ viewed as a subset of $2^X$ in the natural
product topology is called {\em analytic} (see \cite{Kechr}) if it is a
continuous image of the irrationals $\N^\omega$.

A variety of reasons to study {\em analytic spaces}, i.e.~countable
topological spaces whose topology is analytic is given
in \cite{TU1}, \cite{TU2}, and \cite{T1}. The authors of \cite{TU2}
coined the term {\em effective topology} for the research involving
such spaces and presented a number of questions whose
answers depend on various set-theoretic assumptions in the realm of general
topological spaces (such as the {\em Malykhin problem},
see \cite{hrusak}) that have effective counterparts that can be resolved in
ZFC alone.

Recall that a space $X$ is called {\em sequential} if whenever $\cl{A}\setminus
A\not=\varnothing$ for some $A\subseteq X$ there exists a convergent
sequence $S\subseteq A$ such that $S\to x\in\cl{A}\setminus A$. If $X$
is sequential one naturally defines the {\em sequential closure\/} of
a subset $A$ of $X$ as the set $[A]'$ of all the limits of all the
convergent sequences in $A$. Recursively putting $[A]_0=A$,
$[A]_{\alpha+1}=[[A]_\alpha]'$, and taking unions at the limit
stages, one arrives at the concept of an iterated sequential
closure. It is well known that for any $A\subseteq X$ where $X$ is
sequential, there exists an 
$\alpha\leq\omega_1$ such that $[A]_\alpha=\cl{A}$. This observation
naturally leads to the definition of the {\em sequential order}
$\so(X)$ of a sequential space $X$ as the smallest ordinal
$\gamma\leq\omega_1$ such that $\cl{A}=[A]_\gamma$ for every
$A\subseteq X$. {\em Fr\'echet} spaces are defined as sequential
spaces of sequential order $\leq1$.

In \cite{TU1} S.~Todor\v cevi\'c and C.~Uzc\'ategui show among other
results that a countable
topological group is metrizable if and only if it is analytic and
Fr\'echet thus solving the effective version of Malykhin's question on
the existence of a non metrizable countable Fr\'echet group (the non
effective version of this question was answered in \cite{hrusak}). In
the same paper they pose a question about the sequential orders of
sequential analytic groups which can be considered an effective
version of a question of Nyikos (see \cite{nyikos} and \cite{Sh1}).

In this paper we answer this question by showing that the only
possible sequential orders of sequential analytic groups are $1$ and
$\omega_1$. In addition, we show that such groups have a very well
defined topological structure and their topologies are 
completely described by an ordinal invariant that measures the
scatteredness of their compact subspaces (see below for a more precise
discussion).

We assume that all topological spaces appearing below are regular and
use the standard set-theoretic notation and terminology
(see \cite{kunen} and \cite{Kechr}). We proceed by defining some of the less common
concepts.

A space $X$ is called a $k_\omega$-space if there exists a $\set
K_n:n\in\omega.$ where each $K_n$ is a compact subset of $X$ such that
a set $U\subseteq X$ is open if and only if each $U\cap K_n$ is
relatively open. The class of $k_\omega$-spaces is stable under taking
products, i.e.\ the product of two (or any finite number of)
$k_\omega$-spaces is again $k_\omega$.

Countable $k_\omega$-spaces are sequential and analytic (more
precisely, their topology is $F_{\sigma\delta}$), and form a subclass of
$\aleph_0$-spaces (see \cite{michael}). Instead of the original
definition we shall use the following characterization that describes
$\aleph_0$-spaces in the narrow case when $X$ is sequential.

\begin{lemma}\label{azs}A sequential space $X$ is an
$\aleph_0$-space if and only if there exists a countable collection
$\set A_n:n\in\omega.$ of subsets of $X$ such that for any open
$U\subseteq X$ and any converging sequence $S\subseteq U$ such that
$S\to x\in U$ there is an $n\in\omega$ such that $K_n\subseteq U$ and
$K_n\cap S$ is infinite.
\end{lemma}

Sequential $\aleph_0$-spaces are exactly the quotient images of
separable metric spaces (\cite{michael}). Even in the case of a
countable $X$, not every $\aleph_0$-space is a $k_\omega$-space,
however, when $X$ is a countable sequential non Fr\'echet topological group, a
corollary of a more general result in \cite{tbana} implies that $X$
is a $k_\omega$-space if and only if $X$ is an $\aleph_0$-space. A
result in \cite{Sh2} shows that for each such group $\so(X)=\omega_1$. Perhaps the most
surprising property of the class of all $k_\omega$ countable group topologies is
that there are exactly $\omega_1$ of them, moreover, the topological
type of such group is uniquely described by the supremum of {\em
Cantor-Bendixson ranks} of its compact subspaces (see \cite{Z} and \cite{kannan}). 

Recall that a collection of open subsets of a topological space is
called a {\em $\pi$-base} if every open subset of the space contains a
member of the collection. Furthermore, a collection of open subsets is
called a {\em local $\pi$-base at $x\in X$} if every neighborhood of
$x$ contains a set in the collection. It is an easy observation that a
collection of open subsets of $X$ that is a local $\pi$-base at every
point in some dense subset of $X$ is a $\pi$-base of $X$. The
following lemma is well known (the second statement is the famous
Birkhoff-Kakutani metrization theorem).

\begin{lemma}\label{pib}Every topological group with a countable
local $\pi$-base at any point is first countable and every first countable topological
group is metrizable.
\end{lemma}

The {\em countable sequential fan $S(\omega)$} is defined as the set
$\omega^2\cup\{\omega\}$ equipped with the topology in which every
$(n,i)\in\omega^2$ is isolated and the basic neighborhoods of $\omega$
are $U_f=\set(n,i):i\geq f(n).$ where $f:\omega\to\omega$.

\begin{definition}\label{pQ}Let $X$ be a topological space. Let $x\in
X$ and $\seq D_n$ be a collection of infinite countable
closed discrete subsets of $X$ such that for every open $U\subseteq X$
such that $x\in U$ there exists an $n\in\omega$ such that $U\cap D_n$ is infinite. Then
$Y=\cup\set D_n:n\in\omega.\cup\{x\}$ is called a {\em \pQ\ subspace
of $X$}.
\end{definition}
The utility of the previous definition is illustrated by the following lemma.
\begin{lemma}\label{pQns}If $X$ has a \pQ\ subspace, $X\times
S(\omega)$ is not sequential.
\end{lemma}
\begin{proof}
Let $Y=\cup\set D_n:n\in\omega.\cup\{x\}$ be a \pQ\ subset of $X$
where $x\in X$ and $D_n=\set d^n_i:i\in\omega.$ be as in
Definition~\ref{pQ}. Define 
$$
A=\cup\set\cup{\set
(d^n_i,(n,i)):i\in\omega.}:n\in\omega.\subseteq X\times
S(\omega).
$$ 
Now $(x,\omega)\in\cl{A}\setminus A$ but there is no infinite
$S\subseteq A$ such that $S\to y$ for some $y\in X$. Indeed, otherwise
the projection $\pi_2(S)\subseteq S(\omega)$ contains an infinite
`diagonal' convergent subsequence in $S(\omega)$ or one of the closed
discrete subspaces $\set(d^n_i,(n,i)):i\in\omega.$ contains an
infinite convergent subsequence.
\end{proof}

The following lemma is a corollary of Lemma~16 and Corollary~2
of \cite{Sh1}.

\begin{lemma}\label{nwdG}Let $\tau$ be a sequential group topology on
$\N$ such that $\so(\tau)<\omega_1$. If $\set N_i:i\in\omega.$ is a
collection of nowhere dense subsets of $\N$ there exists an
$S\subseteq\N$ such that $S\to x$ for some $x\in\N$ and each $S\cap
N_i$ is finite.
\end{lemma}

\section{Analytic and other classes of spaces}

In the arguments below, we shall assume that $\tau$ stands for some
analytic topology on a countable set $X$. To simplify the notation we
will assume that $X=\N$ whenever it is convenient. All the references to
topological operations and properties such as convergence, etc.\ are
relative to this topology.

We shall also fix a subtree $T$ of $\tNN$ (see \cite{T1} for the
definition of the tree order) that defines $\tau$, i.e.\ 
such that $U=\pi_1(f)$ for some branch $f$ of $T$ whenever $U\in\tau$
is infinite. Given a
$P\subseteq2^\N$ and a $\sigma\in T$ we will use the notation 
$$
\fc\sigma|P=\cap\set\pi_1(f):f\hbox{ is a branch of }T\hbox{
  that extends $\sigma$ such that }\pi_1(f)\in P.
$$
if such $f$ exist; otherwise we put $\fc\sigma|P=\varnothing$.

\begin{lemma}\label{dom}Let $P\subseteq2^\N$ and $S\to x$ for some
  infinite $S\subseteq\N$. Suppose there is an
  open $U\ni x$ such that $U\in P$. Then there exists a $\sigma\in T$
  such that $S\ain\fc\sigma|P$.
\end{lemma}
\begin{proof}
Let $f$ be a branch of $T$ such that $\pi_1(f)=U$. Pick $\sigma_{-1}\in
T$ such that $f$ extends $\sigma_{-1}$ and $x\in\pi_1(\sigma_{-1})$. Suppose
no $\sigma\in T$ with the property stated in the lemma exists. Using
this one can inductively construct a sequence $\seq \sigma_i$ of
elements of $T$ and a strictly increasing sequence $\seq n_i\subseteq
S$ such that for every $i\in\omega$
\begin{itemize}
\item[(1)] $\sigma_{i+1}$ extends $\sigma_i$;

\item[(2)] $n_i<\max\pi_1(\sigma_i)$ and $n_i\not\in\pi_1(\sigma_i)$;

\item[(3)] there exists a branch $f_i$ of $T$ that extends $\sigma_i$
  such that $\pi_1(f_i)\in P$.

\end{itemize}
Put $f_{-1}=f$. Let $i\in\omega\cup\{-1\}$ and note that (3) holds for
$i=-1$. By the assumption and~(3)
$S\not\ain\fc\sigma_i|P\not=\varnothing$ so one can pick a branch $f_{i+1}$ of
$T$ that extends $\sigma_i$ such that $n_{i+1}\not\in\pi_1(f_i)$ for
some $n_{i+1}\in S$, $n_{i+1}>n_i$. Let $\sigma_{i+1}$ be such that
$f_{i+1}$ extends $\sigma_{i+1}$ and $\max\sigma_{i+1}>\max\sigma_i$,
$\max\sigma_{i+1}>n_{i+1}$. 

Put $f'=\cup\seq\sigma_i$. Then $S\not\ain\pi_1(f')\ni x$, a
contradiction.
\end{proof}

As usual, a set function $F:2^X\to2^X$ will be called {\em monotone\/}
if $F(A)\subseteq F(B)$ whenever $A\subseteq B$.
\begin{lemma}\label{thin}Let $\set Q_\alpha:\alpha\in\omega_1.$ be such that each
  $Q_\alpha\subseteq2^\N$ and $Q_\beta\subseteq Q_\alpha$ when
  $\beta\leq\alpha$. Put $P_\alpha=\set
  B\subseteq\N:q\setminus F(B)\not=\varnothing\hbox{ for every }q\in
  Q_\alpha.$ where $F:2^\N\to 2^\N$ is a monotone set function. 
  Then there exists a $\gamma\in\omega_1$ such that
  $q\not\subseteq F(\fc\sigma|{P_{\gamma'}})$ for any $q\in\cup\set
  Q_\alpha:\alpha\in\omega_1.$ and any $\gamma'\geq\gamma$.
\end{lemma}
\begin{proof}
Since $T$ is countable, it is enough to show that
$\fc\sigma|{P_\alpha}=\varnothing$ whenever there is a $q\in
Q_\alpha$ such that $q\subseteq F(\fc\sigma|{P_\beta})$ for some
$\beta<\omega_1$. Assuming such a
$q$ exists, suppose $\fc\sigma|{P_\alpha}\not=\varnothing$. Then
$\alpha>\beta$, otherwise $q\in Q_\alpha\subseteq Q_\beta$ and
$F(\fc\sigma|{P_\beta})\setminus q\not=\varnothing$, since
each $P_\beta$ is closed under taking subsets, which contradicts $q\subseteq
F(\fc\sigma|{P_\beta})$. 

There exists a branch $f$ of $T$ that extends $\sigma$ such that
$\pi_1(f)\in P_\alpha\subseteq P_\beta$ so
$q\setminus F(\pi_1(f))\not=\varnothing$. Now
$F(\pi_1(f))\supseteq F(\fc\sigma|{P_\beta}\cap\pi_1(f))=F(\fc\sigma|{P_\beta})\supseteq
q$, a contradiction.
\end{proof}

\begin{lemma}\label{pialt}Let $\tau$ be an analytic sequential topology on
  $\N$. Then there exists a countable family
  $\cal U$ of open in $\tau$ sets and a countable
  family\/ $\Xi$ of nowhere dense subsets of $\N$ such that at least one of the
  following alternatives holds for every $x\in\N$:
  \begin{itemize}
  \item[\rm(1)]for any infinite sequence $S\subseteq\N$ such that
  $S\to x$ there is a $\xi\in\Xi$ such that $S\ain\xi$;

  \item[\rm(2)]$\cal U$ is a local $\pi$-base at $x$.

  \end{itemize}
\end{lemma}
\begin{proof}
Put $F(B)=\cl{B}$ and define $Q_\alpha=\set
\Int(\cl{\fc\sigma|{P_\beta}}):\sigma\in T,\ \beta<\alpha.$ where $P_\beta$ is defined as in
Lemma~\ref{thin}. Find $\gamma\in\omega_1$ as in
Lemma~\ref{thin}. It follows from the construction of $Q_\alpha$ that
every $\fc\sigma|{P_\gamma}$ is nowhere dense.

Put $P=P_\gamma$, $\Xi=\set\fc\sigma|{P_\gamma}:\sigma\in T.$, ${\cal
U}=Q_\gamma$, and let
$S\to x$ for some $x\in\N$. If $\cal U$ is not a local $\pi$-base at
$x$ there
exists an open $U\ni x$ such that $q\setminus\cl{U}\not=\varnothing$ for
every $q\in Q_\gamma$. Thus $U\in P_\gamma$ and Lemma~\ref{dom} implies
that there is a $\sigma\in T$ such that
$S\ain\fc\sigma|{P_\gamma}\in\Xi$.
\end{proof}
The following is an immediate corollary of Lemma~\ref{nwdG},
Lemma~\ref{pib}, and the lemma above. It answers Question~7.1
from~\cite{TU1}. Theorem~\ref{Mresu} below together with a result
in \cite{Sh2} can also be used to obtain this statement.
 
\begin{corollary}A countable topological group is metrizable if and only
if it has a sequential analytic topology with the sequential order
less than $\omega_1$.
\end{corollary}

A lemma in \cite{DB} and a simple argument result in the following
corollary to Lemma~\ref{pialt}.

\begin{corollary}\label{piFrech}Every analytic Fr\'echet space has a
countable $\pi$-base.
\end{corollary}
\begin{proof}
Let $X$ be an analytic Fr\'echet space and put $U=X\setminus\cl{P}$
where $P$ is the set of all isolated points of $X$. Observe that a
result in \cite{DB} shows that the first alternative of
Lemma~\ref{pialt} does not hold in Fr\'echet spaces without isolated
points so $U$ has a countable $\pi$-base $\cal B$. Adding all the
singletons from $P$ to $\cal B$ one obtains a countable $\pi$-base for $X$.
\end{proof}

A similar proof shows that the conclusion of Lemma~\ref{pialt} can be
sharpened for homogeneous spaces.

\begin{corollary}\label{dihom}Let $X$ be a homogeneous analytic
sequential space. Then $X$ has either a countable $\pi$-base or a
countable collection $\Xi$ of nowhere dense subsets such that
property~(1) of Lemma~\ref{pialt} holds at every $x\in X$.
\end{corollary}

A quick observation reveals that a disjoint union $\Q\cup S_\omega$ of
a copy of the rationals and the {\em Arkhangel'skii-Franklin space
$S_\omega$} (see \cite{TU1} for a nice definition of $S_\omega$ and
further references) does not satisfy the dichotomy of
Corollary~\ref{dihom}. Therefore the restrictions in the corollaries
above cannot be removed.

\begin{lemma}\label{pQseq}
Let $X$ be an analytic sequential space. Then $X$ is
a $k_\omega$-space or there exists a \pQ\ subspace of $X$.
\end{lemma}
\begin{proof}
Put $F(B)=\cl{B}$ and define
$Q_\alpha=\set\cl{\fc\sigma|{P_\beta}}:\sigma\in T, \hbox{
$\cl{\fc\sigma|{P_\beta}}$ is not compact}.$ where $P_\beta$ is defined
as in Lemma~\ref{thin}. Let $\gamma\in\omega_1$ be as in Lemma~\ref{dom}. The
construction of $Q_\alpha$ implies that every $\cl{\fc\sigma|{P_\gamma}}$
is compact.

Suppose $X$ has no \pQ\ subspace. 
Put $P=P_\gamma$, define a countable family 
$K=\set\cl{\fc\sigma|{P_\gamma}}:\sigma\in T, \hbox{
$\cl{\fc\sigma|{P_\gamma}}$ is compact}.$, and let $S\to x$ for some
$x\in X$. For each $q\in Q_\gamma$ pick a closed infinite discrete
subset $D_q\subseteq q$. Call the collection just constructed $\cal
D$. The case of finite $\cal D$ is immediate so we can assume that
$\cal D$ is infinite. Since $\cup{\cal D}\cup\{x\}$ is not a \pQ\
subspace of $\N$ there exists an open $U\ni x$ such that
$D_q\setminus\cl{U}\subseteq q\setminus\cl{U}\not=\varnothing$ for
every $q\in Q_\gamma$. Thus $U\in P_\gamma$ and it follows from
Lemma~\ref{dom} that there exists a $\sigma\in T$ such that
$S\ain\fc\sigma|{P_\gamma}\subseteq\cl{\fc\sigma|{P_\gamma}}\in K$.
\end{proof}

The next corollary follows from the lemma above and Lemma~\ref{nwdG}.

\begin{corollary}Let $X$ be an analytic space. Then $X\times
S(\omega)$ is sequential if and only if $X$ is a $k_\omega$-space.
\end{corollary}

Let $C$ be a closed copy of $S(\omega)$ in $X$. If $X$ is a
topological group, it is convenient to assume that $\Xo$ is the limit
point of $C$ and write $C=\cup\set C_n:n\in\omega.\cup\{\Xo\}$ where $C_n=\seq
c^n_i\to\Xo$ are disjoint subsets of $X$ that do not contain $\Xo$, such that each
$A\subseteq\cup C_n$ satisfying $|A\cap C_n|<\omega$ for every
$n\in\omega$ is closed in $X$. We will refer to this representation of
$C$ as a {\em natural closed copy of $S(\omega)$ in $X$} and will use the
notation above for the sake of brevity below. 

\begin{lemma}\label{Sthin}
Let $X$ be an analytic group and let $\cup\set
C_n:n\in\omega.$ be a natural closed copy of $S(\omega)$ in $X$. There
exists a countable family $\Xi$ of subsets of $X$ with the following
properties:
\begin{itemize}
\item[\rm(1)]for each $p\in \Xi$ there exists an $M_p\in\omega$ such that
$|p\cap a\cdot C_n|=\omega$ implies $n\leq M_p$ for any $a\in X$;

\item[\rm(2)]for each infinite $S\subseteq X$ where $S\to x$ for some $x\in
X$ there exists a $p\in \Xi$ such that $S\ain p$.

\end{itemize}
\end{lemma}
\begin{proof}
Put $F(B)=(\cl{B})^{-1}\cl{B}$ and
define 
$$
Q_\alpha=\set
F(\fc\sigma|{P_\beta}):\sigma\in T,\ \beta<\alpha,\ |{\set
n\in\omega:F(\fc\sigma|{P_\beta})\cap
C_n\not=\varnothing.}|=\omega.
$$
for $\alpha<\omega_1$, where $P_\beta$ is as in Lemma~\ref{thin}.

Let $\gamma<\omega_1$ be the index provided by Lemma~\ref{thin} and
$\Xi=\set \fc\sigma|{P_\gamma}:\sigma\in T.$. Note that for $p=\fc\sigma|{P_\gamma}\in \Xi$ the set
$\set n\in\omega:F(p)\cap C_n\not=\varnothing.$ is finite. Otherwise
$q=F(p)\in Q_{\gamma+1}$ contrary to the choice of $\gamma$. Pick
$M_p\in\omega$ so that $F(p)\cap C_n=\varnothing$ for $n\geq M_p$. Now if
$|\cl{p}\cap a\cdot C_n|=\omega$ for some $n\in\omega$ and $a\in X$ then
$a\in\cl{p}$ thus $F(p)\cap C_n=(\cl{p})^{-1}\cl{p}\cap C_n\supseteq
a^{-1}\cdot\cl{p}\cap C_n\not=\varnothing$ so $n\leq M_p$.

Let $S\to x\in X$. Put $P=P_\gamma$. One can construct a set
$D\subseteq\cup\set C_n:n\in\omega.$ by induction such that $|D\cap
C_n|\leq 1$ for each $n\in\omega$ and $D\cap q\not=\varnothing$ for
each $q\in Q_\gamma$. Note that $D$ is a closed discrete subset of $X$ and
$x^{-1}x=\Xo\not\in D$. Therefore there exists an open neighborhood $U$ of
$x$ such that $F(U)\cap D=(\cl{U})^{-1}\cl{U}\cap D=\varnothing$ and
thus $U\in P$. Now Lemma~\ref{dom} implies the existence of a
$\sigma\in T$ such that $S\ain\fc\sigma|P\in \Xi$.
\end{proof}

\begin{lemma}\label{pQneg}Let $X$ be an analytic non Fr\'echet group. If $X$
contains a \pQ\ subspace then $X$ is not sequential.
\end{lemma}
\begin{proof}
Suppose $X$ is sequential. Since $X$ is not Fr\'echet, $X$ contains a
closed copy of $S(\omega)$ (see \cite{nyikos}) so let $\cup\set
C_n:n\in\omega.$ be a natural closed copy of $S(\omega)$ in $X$. 
Let $\Xi=\set p_n:n\in\omega.$ be the family provided
by Lemma~\ref{Sthin} and put $M^n=\max\set M_i:i\leq n.+n+1$ where $M_i$ have
the property of Lemma~\ref{Sthin}(2).

Let $D=\cup\set D_n:n\in\omega.\cup\{\Xo\}$
be a \pQ\ subspace of $X$ where $D_n=\set d^n_i:i\in\omega.$ are disjoint closed discrete
subspaces of $X$. Now $M^n$ is a strictly increasing sequence and by the construction
the set
$a\cdot C_{M^n}\cap(\cup\set p_i:i\leq n.)$ is finite for every
$n\in\omega$ and $a\in X$. Pick a strictly increasing sequence $\seq
r^n_i\subseteq\omega$ such that $e^n_i=d^n_i\cdot
c^{M^n}_{r^n_i}\not\in\cup\set p_k:k\leq n.$ and all $e^n_i\not=\Xo$ are
distinct. Define $E=\set e^n_i:n,i\in\omega.$ and suppose there is an
infinite $S\subseteq E$ such that $S\to x$ for some $x\in X$.

Note that for every $n\in\omega$ the set $\set
e^n_i:i\in\omega.\subseteq D_n\cdot C_{M^n}$ is closed and discrete in
$X$ so we can assume that $S=\set e^{n_i}_{m_i}:i\in\omega.$ for some
strictly increasing $\seq n_i\subseteq\omega$. Now $S\ain p_n$ for
some $p_n\in \Xi$ by Lemma~\ref{Sthin}(2) but $e^{n_i}_{m_i}\not\in p_n$
for $n_i>n$, a contradiction. Thus $E$ is sequentially closed.

Let $V$ be an open neighborhood of $\Xo$ in $X$ and $U\ni\Xo$ be an
open subset of $X$ such that $U\cdot U\subseteq V$. Pick $n\in\omega$
such that $D_n\cap U$ is infinite and choose $k\in\omega$ large enough
so that $c^{M^n}_{r^n_k}\in U$ and $d^n_k\in D_n\cap U$. Now
$e^n_k=d^n_k\cdot c^{M^n}_{r^n_k}\in U\cdot U\subseteq V$. Thus
$\Xo\in\cl{E}\setminus E$, a contradiction.
\end{proof}
 
Lemmas~\ref{pQneg} and \ref{pQseq} imply the following theorem

\begin{theorem}\label{Mresu}Let $X$ be a countable sequential group. Then the
following are equivalent:
\begin{itemize}
\item[\rm(1)]the topology of $X$ is analytic;

\item[\rm(2)]the topology of $X$ is $F_{\sigma\delta}$;

\item[\rm(3)]$X$ is either first countable or $k_\omega$.

\end{itemize}
\end{theorem}

The result of Zelenyuk (see \cite{Z}) mentioned in the introduction
together with the theorem above imply

\begin{corollary}There are exactly $\omega_1$ non homeomorphic
analytic sequential group topologies. Moreover, if $X$ is an infinite analytic
sequential group then all finite powers $X^n$ are such and are
homeomorphic to each other.
\end{corollary}

\section{Examples and questions}
It has been demonstrated by a number of authors that sequential
$\aleph_0$-spaces have a number of properties resembling those of
separable metric spaces.

On the other hand, the properties that diferentiate between the two
classes of spaces are strikingly similar to those that separate
analytic spaces and countable metrizable ones. As an example, it is
easy to show that a sequential $\aleph_0$-space with a {\em weak
diagonal sequence property\/} is first countable (see \cite{michael}
and \cite{TU2}). The following example shows that analytic sequential
spaces are not necessarily $\aleph_0$-spaces, thus the statement of
Theorem~\ref{Mresu} is limited to topological groups.

\begin{example}{\rm Consider the following basis for a topology on $\tN$,
viewed as a tree with the usual order of end extension. Every point
$x\in\tN\setminus\{\varnothing\}$ is isolated and the basis of
neighborhoods of $\varnothing$ consists of complements of finite
unions of branches together with $\varnothing$. It is shown
in \cite{TU2}, Example~5.6 (see also \cite{TU1}, Remark~4.8) that the
resulting topology is $F_\sigma$ and Fr\'echet with the weak diagonal
sequence property but not first countable. Thus the space constructed
is not an $\aleph_0$-space.}
\end{example}

A partial result going in the opposite direction is possible. 
The next lemma is an easy corollary of the result that each quotient
image of the rationals is {\em determined} (see \cite{michael}) by a
countable family of metrizable subspaces.
\begin{lemma}If $X$ is a quotient image of a countable metric space
(equivalently, the rationals $\Q$) then $X$ is analytic (more
precisely, $X$ has an $F_{\sigma\delta}$ topology).
\end{lemma}

One might hope that due to their `tame' convergence structure,
$\aleph_0$-spaces form a subfamily of all analytic spaces. The next
simple example shows that this is not the case. 
\begin{example}
There are $2^\C$ countable Fr\'echet\/ $\aleph_0$-spaces with a single
non isolated point. In particular, there are non analytic spaces of
such kind.
\end{example}
\begin{proof}
Let $A\subseteq\R$ be an arbitrary subset of the real line. Put
$$
M(A)=(\{0\}\times A)\cup\set (1/n,q):n\in\N,q\in\Q.
$$
Define a topology on $M(A)$ by making the Euclidean neighborhoods of
points $(0,a)$, $a\in A$ the new basic neighborhoods and making all
other points isolated. The space just constructed is a separable
metrizable one. Consider the quotient map that sends $\{0\}\times A$
into a single point $\infty$ and is 1-1 on the rest of $M(A)$. Its image is a
Fr\'echet $\aleph_0$-space $P(A)$ with a single non isolated point.

Let $A\subseteq\R$ and $B\subseteq\R$ be two different subsets of the
real line, and let, say, $a\in A$ be such that $a\not\in B$. Given any
sequence of rationals $\seq q_n\subseteq\Q$ that converges to $a$ the set $\set
(1/n,q_n):n\in\omega.$ is a convergent sequence in $P(A)$ and is a
closed discrete subset of $P(B)$. Hence, the topologies of $P(A)$ and
$P(B)$ differ, resulting in exactly $2^\C$ different Fr\'echet
$\aleph_0$-topologies with a single non isolated point. Noting that
there are at most $\C$ possible homeomorphisms between topologies on a
given countable set, one can pick $2^\C$ pairwise non homeomorphic
spaces $P(A)$. Given that there are at most $\C$ analytic topologies
on any countable set, most $P(A)$ are non analytic.
\end{proof}  
\begin{remark}{\rm
The rather crude construction of the example above does not produce any
`explicit' $A\subseteq\R$ such that $P(A)$ is not analytic. A more
precise proof is possible that shows that $A$ is a projection of a
Borel subset of the product of the irrationals and the topology of
$P(A)$ viewed as the subset of the irrationals giving one more control
over the complexity of $P(A)$.}
\end{remark}

The final example shows that the property established in
Lemma~\ref{pialt} is not enough to show that the group is a $k_\omega$-space.

\begin{example}[{\rm CH}]There exists a countable sequential group
$G$ and a countable collection $\Xi$ of nowhere dense subsets of $G$
such that $G$ is not a $k_\omega$-space and for every convergent
sequence $S\subseteq G$ there is a $\xi\in\Xi$ such that
$S\subseteq\xi$.
\end{example}
\begin{proof}
The full details of the construction are somewhat tedious and are of limited
interest. We therefore present just a sketch of the proof. A number of
similar arguments can be found in \cite{shi}.

One starts with a non discrete first countable topology $\tau_0$ on
$\Q$ (any topologizable countable group would suffice; it is easy to
see that a similar construction just as readily gives an example of a
topological {\em field} with these properties). 
Pick a compact subset $K$ of $\Q$ such that $0\in K$ has
Cantor-Bendixson rank $\omega$ in $K$. Pick a countable collection of
convergent sequences in $\tau_0$ that witness the Cantor-Bendixson
rank of each point of $K$. Let $\eta_0$ be the finest group topology on
$\Q$ in which each of these sequences converges. The existence of such
a topology can be established by an easy argument 
(see, for example \cite{shi}). Let $\set A_\alpha:\alpha\in\omega_1.=2^\Q$.

The construction proceeds by induction on $\alpha\in\omega_1$ where at
stage $\alpha$ one defines a pair of topologies
$\tau_\alpha\subseteq\eta_\alpha$ such that $\tau_\alpha$ is first
countable and $\eta_\alpha$ is determined by countably many compact
subsets of finite Cantor-Bendixson rank. At limit stages the
construction proceeds in a natural (and trivial) way. 

At stage $\alpha+1$ one picks $\tau_{\alpha+1}\supseteq\tau_\alpha$
such that $\tau_{\alpha+1}$ contains enough open in $\eta_\alpha$ subsets to show that
a given $A_\alpha\subseteq\Q$ is closed in $\tau_{\alpha+1}$ provided
it is closed in $\eta_\alpha$ and not compact in $\tau_{\alpha+1}$
provided it is not compact in $\eta_\alpha$.
Now $\eta_{\alpha+1}\subseteq\eta_\alpha$ is chosen as the finest topology coarser than
$\eta_\alpha$ in which $S\to0$ for some $S\subseteq K$ such
that $S\to0$ in $\tau_{\alpha+1}$ and $S$ is discrete in
$\eta_\alpha$. Such an $S$ can be built inductively by first finding
a compact (in $\tau_{\alpha+1}$) $K'\subseteq K$ of infinite
Cantor-Bendixson rank.

Define
$$
\tau=\cup\set\tau_\alpha:\alpha\in\omega_1.=\cap\set\eta_\alpha:\alpha\in\omega_1.
$$
and put
$$
\Xi=\set d+(-1)^{\delta_0}K+\cdots +(-1)^{\delta_n}K:\delta_i\in\{0,1\}, d\in[\Q]^{<\omega},
n\in\omega.
$$
It is easy to see that the choice of $\tau_\alpha$
ensures that $\tau$ is sequential and the choice of $\eta_\alpha$ and $\tau_{\alpha+1}$
prevents $\tau$ from being $k_\omega$. Moreover, each compact in $\tau$
subset of $\Q$ is compact in some $\eta_\alpha$ and therefore resides
in some `monomial' over $K$. Thus the family $\Xi$ has the desired
property.
\end{proof}
Finally, it seems natural to ask whether Lemma~\ref{pQneg} can be
generalized to non analytic groups.

\begin{question}Do there exist (countable) sequential non Fr\'echet
groups that contain a (closed) \pQ\ subspace?
\end{question}


\begin{thebibliography}{99}
\bibitem{tbana}T.~Banakh, L.~Zdomskyy, {\em The topological structure
  of (homogeneous) spaces and groups with countable $cs^*$-network},
  App.\ Gen.\ Top.\ {\bf 5} (2004), no.~1, pp.~25--48

\bibitem{DB}D.~Barman and A.~Dow, {\em Proper forcing axiom and
  selective separability}, Topology Appl.\ {\bf 159} (2012),
  pp.~806--813

\bibitem{vanD} E.K.~van Douwen, {\em The integers and Topology}, in: K.Kunen, J.E.Vaughan (eds.), Handbook
of Set-Theoretic Topology (North-Holland, Amsterdam, {\bf 1984}), pp.~111--167.

\bibitem{hrusak}M.~Hru\v s\'ak and U.A.~Ramos-Garc\'ia, {\em
  Malykhin's problem}, preprint

\bibitem{kannan}V.~Kannan, {\em Ordinal Invariants in Topology},
  Memoirs of the Amer.\ Math.\ Society {\bf 32} (1981) no.~245

\bibitem{Kechr}A.S.~Kechris, {\em Classical descriptive set theory},
Springer, {\bf 1995}

\bibitem{kunen}K.~Kunen, {\em Set theory: An Introduction to
  Independence Proofs}, Vol.~{\bf 102} of Studies in Logic and the
  Foundations of Mathematics, North-Holland, Amsterdam 1980

\bibitem{michael}E.~Michael, $\aleph_0$-spaces, J.\ Math.\ Mech.\ {\bf
  15} (1966), pp.~983--1002

\bibitem{nyikos}P.~J.~Nyikos, {\em Metrizability and the
  Fr\'echet-Urysohn property in topological groups},
  Proc.\ Amer.\ Math.\ Soc.\ {\bf 83} (1981), pp.~793--801

\bibitem{pierone}R.~Peirone, {\em Regular semitopological groups of
  every countable sequential order}, Topology Appl.\ {\bf 58} (1994),
  pp.~145--149

\bibitem{Sh1}A.~Shibakov, {\em No interesting sequential groups},
  preprint

\bibitem{shi}A.~Shibakov, {\em Sequential topological groups of any
  sequential order under CH}, Fund.\ Math.\ {\bf 155} (1998) no.~1,
  pp.~79--89

\bibitem{Sh2}A.~Shibakov, {\em Metrizability of sequential topological groups
with point-countable $k$-networks}, Proc.\ Amer.\ Math.\ Soc.\ {\bf
  126} (1998), pp.~943--947

\bibitem{T1}S.~Todor\v cevi\'c, {\em Analytic gaps}, Fund.\ Math.\
  {\bf 150} (1996), pp.~55--66

\bibitem{TU1}S.~Todor\v cevi\'c and C.~Uzc\'ategui, {\em Analytic
  $k$-spaces}, Topology Appl.\ {\bf 146--147} (2005), pp.~511--526

\bibitem{TU2}S.~Todor\v cevi\'c and C.~Uzc\'ategui, {\em Analytic
  topologies over countable sets}, Topology Appl.\ {\bf 111} (2001),
  pp.~299--326

\bibitem{Z}E.~Zelenyuk, {\em Topologies on groups determined by
  compact subspaces}, Matem.\ Studii {\bf 5} (1995), pp.~5--16 (in Russian).
\end{thebibliography}
\end{document}